\documentclass{amsart}
\usepackage{amsmath,amssymb,epsf,ifthen,graphics,graphicx,url}

\usepackage{tikz}
\usetikzlibrary{arrows,backgrounds,decorations}

\newtheorem{thm}{Theorem}
\newtheorem{lem}[thm]{Lemma}
\newtheorem{alg}[thm]{Algorithm}
\newtheorem{cor}[thm]{Corollary}
\newtheorem{prop}[thm]{Proposition}
\theoremstyle{definition}
\newtheorem{defn}[thm]{Definition}

\newcommand{\blackboard}[1]{\ensuremath{\mathbb{#1}}}
\newcommand{\Z}{\blackboard{Z}}

\newcommand{\ba}{\begin{array}}
\newcommand{\ea}{\end{array}}
\newcommand{\be}{\begin{enumerate}}
\newcommand{\ee}{\end{enumerate}}
\newcommand{\bi}{\begin{itemize}}
\newcommand{\ei}{\end{itemize}}

\newcommand{\bc}{\begin{center}}
\newcommand{\ec}{\end{center}}
\newcommand{\bt}{\begin{tabular}}
\newcommand{\et}{\end{tabular}}

\newcommand{\inv}{^{-1}}

\newcommand{\LiptonZ}{MR0445901}

\newcommand{\DKL}{DiekertKL12}
\newcommand{\LohreyOndrusch}{MR2340901}

\newcommand{\LS}{MR0577064}

\newcommand{\Simon}{MR563704}
\newcommand{\Parry}{MR1062874}
\newcommand{\Wehrfritz}{Wehrfritz}

\newcommand{\Robinson}{MR1357169}
\newcommand{\Waack}{Waack}

\newcommand{\Mihailova}{MR0222179}
\newcommand{\Segal}{Segal}

\newcommand{\Mias}{MR2437984}

\newcommand{\Gromov}{MR623534}

\newcommand{\MKS}{MR2109550}
\newcommand{\SapirEtAl}{MR1933724}

\begin{document}

\title[Groups that have normal forms computable in logspace]{On groups that have normal forms \\computable in logspace}

\author[M. Elder] {Murray Elder}\address{School of Mathematical and Physical Sciences, The
University of Newcastle,
Callaghan NSW 2308 Australia}
\email{Murray.Elder@newcastle.edu.au}

\author[G. Elston] {Gillian Elston} \address{Department of
  Mathematics, Hofstra University, Hempstead NY 11549 USA}
\email{Gillian.Z.Elston@hofstra.edu}

\author[G. Ostheimer] {Gretchen Ostheimer} \address{Department of
  Computer Science, Hofstra University, Hempstead NY 11549 USA}
\email{Gretchen.Ostheimer@hofstra.edu}

\keywords{Logspace algorithm; logspace normal form; logspace embeddable; wreath product; Baumslag-Solitar group; logspace word problem}
\subjclass[2010]{20F65; 68Q15}

 \date{\today}

\thanks{The first author acknowledges support from the Australian Research Council grants DP110101104, DP120100996 and FT110100178, and is grateful for the hospitality of Hofstra University.}

\begin{abstract}
We consider the class of finitely generated groups which have a normal form computable in {\em logspace}.
We prove that the class of such groups
 is closed
under passing to finite index
subgroups, direct products, wreath products, and certain free products and infinite extensions, and includes the
solvable Baumslag-Solitar groups,
as well as  non-residually finite (and hence non-linear) examples.
We define a group to be {\em logspace
  embeddable} if it embeds in a group with normal forms computable in
logspace. We prove that  finitely generated nilpotent groups are logspace embeddable.
It follows that all groups of polynomial growth are logspace embeddable.
\end{abstract} \maketitle

\section{Introduction}\label{sec:intro}
Much of combinatorial, geometric and computational group theory
focuses on computing efficiently in finitely generated groups.
Recent work in group-based cryptography  demands fast and
memory-efficient ways to compute {\em normal forms} for group
elements \cite{\Mias}. In this article we consider groups which have a normal form
over some finite generating set, for which there is an
algorithm to compute the normal form of a given input word in  {\em logspace}.
We show that the class of finitely generated groups having a logspace normal
form is surprisingly large.


\begin{defn}\label{def:transducer}
A {\em deterministic logspace transducer} consists of a finite state
control and three tapes: the first {\em input tape} is read only, and
stores the input word; the second {\em work tape} is read-write, but is
restricted to using at most $c\log n$ squares, where $n$ is the length
of the word on the input tape and $c$ is a fixed constant; and the
third {\em output tape} is write-only, and is restricted to writing left
to right only. A transition of the machine takes as input a letter of the
input tape, a state of the finite state control, and a letter on the
work-tape. On each transition the machine can modify the work tape,
change states, move the input read-head, and write at most a fixed constant number of letters to the
output tape, moving right along the tape for each letter printed.
\end{defn}

Since the position of the read-head of the input tape is an integer between 1 and $n$, we can store it in  binary on the work tape.

\begin{defn}\label{def:logspacecomptable}
Let $X,Y$ be finite alphabets. Let $X^*$ denote the set of all finite length strings in the letters of $X$, including the empty string $\lambda$.
We call $f:X^*\rightarrow Y^*$ a
 {\em logspace computable function}  if there is a deterministic  logspace transducer that on input $w\in X^*$ computes $f(w)$.\end{defn}

\begin{defn}\label{def:nf}
A {\em normal form} $L$ for a group $G$ with finite symmetric generating set $X$ is
any subset of $X^*$ that is in bijection with $G$ under the map
which sends a word $w$ to the group element $\overline{w}$ which it
represents.
\end{defn}

\begin{defn}\label{def:logspacenf}
A logspace computable function
$f:X^*\rightarrow X^*$ for which $f(w)$ is the normal form word for $w$, is called a {\em logspace normal form function} for $(G,X)$.
\end{defn}
\begin{defn}
We say
$(G,X)$ has a {\em logspace  normal form}   if it has a logspace normal form function.\end{defn}

We may sometimes say a normal form is logspace computable without reference to a specific function.

As a simple first example, consider the infinite cyclic group $\langle a \ | \ -\rangle$ which has normal form $\{a^i \ | \ i\in \mathbb Z\}$. Let $f$ be the function that converts a word $w$ in the letters $a^{\pm 1}$ into normal form. Then $f$ can be computed by scanning $w$ from left to write updating a binary counter $i$, stored on the work tape, and when the end of the input is reached, output $a$ $i$ times if $i\geq 0$ or $a^{-1}$ $i$ times if $i<0$. So  the infinite cyclic group has a logspace normal form (with respect to the generating set $\{a,a^{-1}\}$).

 The
{\em word problem} asks for an algorithm for a finitely generated
group which takes as input a word over the generating set, and decides whether or
not the word is equal to the identity in the group.
 In \cite{\LiptonZ}  Lipton and Zalcstein proved that all linear
groups (groups of matrices with entries from a field of characteristic
zero) have word problem solvable in logspace.
Since the class
of linear group includes all free groups and all polycyclic groups,
it follows from their results that the word problem for any such
group can be decided in logspace.  Simon extended this to  linear groups over arbitrary fields \cite{\Simon}.

One might expect the word problem to be computationally easier than computing a normal form. Certainly if one insists on a {\em geodesic} normal form (with respect to some generating set) then this is the case (see
the end of this section).

The purpose of this article is to examine how broad the class of groups with normal forms computable in logspace is.
In Section \ref{sec:gset} we prove that free groups have logspace normal forms,  a result which can be traced back to  \cite{\LohreyOndrusch}. We then
 establish some basic properties of logspace normal forms, including the fact that having a logspace
normal form is independent of finite generating set, and logspace normal forms can be computed in polynomial time.
In Sections \ref{sec:direct}--\ref{sec:wreath}
 we prove that
the class of groups with logspace normal forms is closed
under direct product,  finite index subgroups and supergroups, finite quotients, and wreath product. It follows that finitely generated abelian groups, and the so-called {\em lamplighter groups}, belong to the class. In Sections \ref{sec:freeprod}--\ref{sec:infiniteextensions}
we prove that the class is closed under free product in certain cases and
under certain infinite
extensions, but in both of these contexts we must impose
 restrictions in order for our
proofs to carry through.  In Section \ref{sec:BS} we present a normal form for solvable Baumslag-Solitar groups that can be computed in logspace.
In Section
\ref{sec:logspaceembeddable} we define a group to be {\em logspace
  embeddable} if it is a subgroup of a group with logspace normal
form, and prove various properties about the class of logspace
embeddable groups, and in Section \ref{sec:nilpotent}
we show that finitely generated nilpotent groups  are logspace embeddable.

The problem of logspace {\em geodesic} normal forms is decidably more subtle, and in this article we focus on normal forms that are not necessarily length-minimal.  In Section   \ref{sec:wreath} we show that wreath products such as $\Z\wr \Z^2$ have logspace normal forms, and comment that the problem of computing a geodesic normal form for this group (with respect to the standard generating set) was shown to be NP-hard in \cite{\Parry}, so the existence of a logspace geodesic normal form for this group seems unlikely.
On the other hand many of the normal forms we present here, such as those for free groups described in Proposition \ref{prop:freegp} and free abelian groups described in Corollary \ref{cor:freeabel}, are geodesic. We prove in Proposition \ref{prop:short} that a logspace normal form has length no more than polynomial in the geodesic length.
Recent work of Diekert, Kausch and Lohrey  \cite{\DKL} extends the class of groups with logspace geodesic normal forms to right-angled Artin groups and right-angled Coxeter groups, and gives a partial result for general Coxeter groups. 

It remains to see an example with polynomial time word problem that does not have a logspace normal form.
 Note that  by \cite{\SapirEtAl} a group has word problem in  NP
 if and only if it is a subgroup of a finitely presented group  with polynomial Dehn function.

The authors wish to thank Gilbert Baumslag, Volker Diekert, Arkadius Kalka, Alexei Miasnikov and Chuck Miller
 for very helpful insights and suggestions, and the anonymous reviewer for their careful reading, corrections and suggestions.

\section{Basic examples and properties of logspace normal forms}
\label{sec:gset}

%

%



We begin with a key example of a class of groups with logspace normal form.

Consider the free group $\langle a_1, \dots, a_k \ | \ -\rangle$ of rank $k$  with normal form the set of all freely reduced words over $X=\{a_1^{\pm 1},\dots, a_k^{\pm 1}\}$. An obvious algorithm to convert a word in $X^*$ would be to scan the word and when a canceling pair is read, delete it, step one letter back, and continue reading.  A logspace function  can only {\em read} the input, not write over it, so such an algorithm would not be logspace.  Instead, the following algorithm makes use of the fact that free groups are linear, and so have logspace decidable word problem.

\begin{prop}\label{prop:freegp}
Let $\langle a_1, \dots, a_k \ | \ -\rangle$ be the free group of  finite rank $k$ with normal form the set of all freely reduced words over $X= \{a_i^{\pm 1}\}$. Then there is a logspace computable function $f: X^* \rightarrow X^*$ such that $f(w)$ is the normal form word for $w$.
\end{prop}
\begin{proof}
Fix two binary counters, $\mathrm{c1}$ and $\mathrm{c2}$ and set them both to 1.
\begin{enumerate}
\item  read the letter at position $\mathrm{c1}$ of the input tape (call it $x$).
 \item scan forward to the next $x^{-1}$ letter to the right of position $\mathrm{c2}$, and set $\mathrm{c2}$ to be the position of this $x^{-1}$.
 \item input the word from position $\mathrm{c1}$ to $\mathrm{c2}$ on the input tape into the logspace word-problem function for  the free group of rank $k$.
 \begin{itemize}
 \item if this function returns
  {\em trivial}, output nothing, set $\mathrm{c1}= \mathrm{c2}+1$, and return to Step (1).
  \item  if it returns  {\em non-trivial},  return to Step  (2).
  \end{itemize}
\item   if there is no next $x^{-1}$ letter, write $x$ to the output tape, set $\mathrm{c1}=\mathrm{c1}+1$, and return to Step (1).
 \end{enumerate}
 In other words, the algorithm reads $x$ and looks for a
 subword $xux^{-1}$ where $u$ evaluates to the identity. If it finds such a subword, it effectively cancels it by not writing it to the output  and moving forward to the next letter after $xux^{-1}$. If there is no such subword starting with $x$, then $x$ will never freely reduce, so it outputs $x$, then repeats this process on the next letter after $x$ on the input tape.
\end{proof}

This  result  can be traced back to  \cite{\LohreyOndrusch}. The proof gives some indication of how one works in logspace. In Proposition \ref{prop:freeClosure} below we prove a more general result, that the class of groups with a normal form computable in logspace is closed under  free products with logspace word problem.

The next lemma shows that logspace computable functions are closed under composition.
\begin{lem}
\label{lem:comp}
If $f, g: X^* \rightarrow X^*$ can both be computed in logspace,
then their composition $f \circ g: X^* \rightarrow X^*$ can also be computed in logspace.
\end{lem}
\begin{proof}
On input a word $w\in X^*$,
run the function $f$ and when $f$ calls for the $j$th input letter, run $g$ on $w$ but instead of outputting, each time $g$ would write a letter, add 1 to a counter (in binary). Continue running $g$ until the counter has value $j-1$, at which point, return the next letter $g$ would output to $f$.
\end{proof}

\begin{lem}
\label{lem:inv}
In a group that has a logspace computable normal form function $f$, the following basic group operations
can be performed in logspace:
\begin{enumerate}
\item we can test whether two words represent the same group element;
\item we can compute a normal form for the inverse of an element.
\end{enumerate}
\end{lem}
\begin{proof}
To test equality, compute $f$ on each word simultaneously and check that successive output letters are identical (without storing them).
To compute the normal form for the inverse of an element, compose $f$ with the (logspace) function that on input $w$, computes the length $n$ of $w$ in binary, then for $i=1$ up to this length returns the formal inverse of the $(n-i+1)$th letter of $w$.
\end{proof}

One might expect that
algorithms using a small amount of space do so at the expense of time,
but it is well-known that this is not the case. To provide further context
for the techniques employed in our proofs, we include here the
standard proof that logspace algorithms run in polynomial time.
\begin{lem} \label{lem:polyTime}
A deterministic logspace algorithm performs at most a polynomial
number of steps.
\end{lem}
\begin{proof}
Define a {\em configuration} of a logspace transducer to be the
contents of the work tape (which includes the position of the input tape read-head), and the current state of the finite state
control. If the work tape has $k$ allowable symbols, and the finite
state control has $d$ states, the total number of distinct
configurations possible on input a word of length $n$ is $dk^{c\log
  n}=O(n^c)$ where $c\log n$ is the maximum number of symbols the work
tape contains. If the machine were to take more than $dk^{c\log n}$ steps, then it would be
in the same configuration twice during the computation, and so would
enter an infinite loop (since the machine is deterministic). The
result follows.
\end{proof}

We next prove that the property of having a logspace normal form
is invariant under change of finite generating sets.

\begin{prop}
\label{prop:indGenSet}
Let $X,Y$ be two finite symmetric generating sets for a group $G$.  If
$(G,X)$ has a logspace normal form, then so does $(G,Y)$.
\end{prop}

\begin{proof}
It suffices to show that adding or deleting a generator does not affect
the existence of a logspace computable normal form function.
Suppose that $Y = X \cup \{ y,y^{-1} \}$, where $y \not
\in X$ and that $w_y \in X^*$ such that $\overline{w_y} = \overline{y}$.
Let $f: Y^* \rightarrow X^*$ be the function that takes a word in $Y^*$
to the word obtained by replacing each occurrence of $y$ with the
word $w_y$, and $y^{-1}$ by the formal inverse of the word $w_y$.
Notice that $\overline{f(u)} = \overline{f(v)}$ if and only if $\overline{u} = \overline{v}$, and that $f$ can be computed in logspace.

We first suppose that $g_X$ is a logspace computable normal form function,
we let $g_Y = g_X \circ f$, and we show that $g_Y$ is a logspace computable normal form function.
By Lemma \ref{lem:comp}, $g_Y$ is logspace computable, so we simply have to establish
that it is a normal form function. Since $f$ maps onto $X^*$, and since $g_X$ is a normal form function,
the natural map from $g_Y(Y^*)$ to $G$ is onto.
Let $u, v \in Y^*$ such that $g_Y(u) = g_Y(v)$.
Since $g_X$ is a normal form function, $\overline{f(u)} = \overline{f(v)}$ and hence $\overline{u}
= \overline{v}$.
Hence the natural map from $g_Y(Y^*)$ to $G$ is injective.
We have shown that $g_Y$ is a normal form function.

We next suppose that $g_Y$ is a logspace computable normal form function,
we let $g_X = f \circ g_Y$. By Lemma \ref{lem:comp}, $g_X$ is logspace computable,
so we only have to show that $g_X$ is a normal form function.
Let $g \in G$. There exists a $v \in g_Y(Y^*)$ such that $\overline{v} = g$.
Therefore, $\overline{f(v)} = g$
and hence the natural map from $g_X(X^*)$ to $G$ is onto.
Let $u, v \in X^*$ such that $g_X(u) = g_X(v)$.
Then $f(g_Y(u)) = f(g_Y(v))$ so $\overline{g_Y(u)} = \overline{g_Y(v)}$.
Since $g_Y$ is a normal form function, $\overline{u} = \overline{v}$.
Thus the natural map from $g_X(X^*)$ to $G$ is injective.
\end{proof}

It will be convenient to assume that the normal form for the identity
element is the empty string.
\begin{prop}\label{prop:emptyWord}
Let $G$ be a group with finite symmetric generating set $X$, and let
$g_X$ be a normal form for $G$ computable in logspace such that $g_X(\lambda)\neq \lambda$.
Define a new normal form $h_X$ for $G$ which is identical to
$g_X$ except that for words representing the identity, $h_X(w) = \lambda$.
Then $h_X$ is logspace computable.
\end{prop}
\begin{proof}
Let $u = g_X(\lambda)$ be of length $m>0$.
Let $f: X^* \rightarrow X^*$ be the map sending $u$ to $\lambda$
and acting as the identity on all other words. Since $m$ is a fixed constant, we can store the word $u$ in a finite state control. Then $f^*$ can be computed in logspace as follows:
using a (binary) counter, scan the input word to compute its length; if it has length $m$, for $i=1$ to $m$, check that the $i$th input letter is identical to the $i$th letter of $u$ (stored in the finite state control); if it is, return $\lambda$, otherwise, move to the start of the input tape and write each letter on the input tape from left to right onto the output tape.
Since $h_X = f \circ g_X$, by Lemma \ref{lem:comp},
$h_X$ is also computable in logspace.
\end{proof}

The next proposition gives a restriction on what types of normal form languages can be
calculated in logspace; namely, the length of the normal form is bounded by a polynomial
in the length of the input.

\begin{prop}
\label{prop:short}
If $G$ has a normal form over $X^*$ which can be computed in logspace,
then there is a constant $c$ such that the normal form for an input
word of length $n$ has length $O(n^c)$.
\end{prop}
\begin{proof}
Let $p$ be the maximum length of a word written to
the output tape in any one transition.  (Note there are a finite
number of possible transitions.)  By Lemma \ref{lem:polyTime}, on input
a word of length $n$, the computation takes a polynomial number of
steps, $O(n^c)$, and in each step at most $p$ letters can be written
to the output tape, so the maximum length of the output normal form
word is $O(pn^c)$.
\end{proof}

\section{Closure under direct product}\label{sec:direct}

\begin{prop}
\label{prop:directprod}
The set of groups with logspace normal forms is closed under direct
product.
\end{prop}
\begin{proof}
Let $G$ and $H$ be groups with symmetric generatoring sets $X$ and $Y$, and
with logspace normal form functions $g_X$ and $h_Y$ respectively.  We
may assume that $X$ and $Y$ are disjoint; let $Z$ be their disjoint
union.   Then $Z$ is a finite set of symmetric
generators for $G\times H$. Define $k_Z: Z^* \rightarrow Z^*$  as follows.
Let $w$ be a word in $Z^*$. Then there exist words $u \in X^*$ and $v
\in Y^*$ such that $w$ consists of $u$ and $v$ interleaved.  We let
$k_Z(w) = g_X(u) h_Y(v)$.  Note that $k_Z(Z^*)$ comprises a unique set
of representatives for $G\times H$.  We can compute $k_Z$ in logspace: read
$w$ once, ignoring all letters from $Y$ and computing $g_X(u)$; read
$w$ again, ignoring all letters from $X$ and computing $g_Y(v)$.
\end{proof}

\begin{cor}\label{cor:freeabel}
All finitely generated abelian groups have logspace normal form
functions.  In the case of $\Z^n$, if $t_1, t_2, \dots, t_n$ is a set
of free generators, the normal forms are of the form $t_1^{\alpha_1}
t_2^{\alpha_2} \dots t_n^{\alpha_n}$ with $\alpha_i\in\Z$.
\end{cor}
\begin{proof}
The result follows from Propositions  \ref{prop:directprod} and \ref{prop:freegp}.
\end{proof}

\section{Closure under passing to finite index subgroups and supergroups}\label{sec:finiteextensions}

Let $G$ and $H$ be finitely generated groups with $G$ a finite index subgroup of $H$.
The goal of this section is to show that $G$ has logspace normal form if and only $H$ does.
To do so, we will show that the standard Schreier rewriting process for $H$ is logspace computable,
and from this our desired result will follow easily. 

We define a {\em rewriting process} for $G$ in the usual way (see, for example, \cite{\MKS}):
\begin{defn} Let $H$ be a group generated by a finite symmetric generating set $Y$.
Let $G$ be a subgroup of $H$, and let $W = \{w_1, w_2, \dots, w_m\}$ be a set of words over $Y$ that generate $G$. 
Let $S$ be the set of words over $Y$ that represent elements of $G$. 
Let $X = \{x_1, x_2, \dots, x_m, x_1^{-1}, x_2^{-1}, \dots, x_m^{-1}\}$ be a new alphabet (disjoint from $Y$), which we take to be a generating set for $G$ via the map that sends $x_i$ to $w_i$ and $x_i^{-1}$ to $w_i^{-1}$ (the formal inverse of $w_i$).
A {\em rewriting process for $G$ with respect to $W$} is a mapping $\tau$ from $S$ to $X^*$ such that 
for all words $u \in S$, $u$ and $\tau(u)$ represent the same element of $G$.
\end{defn}

When $G$ has finite index in $H$, a set $W$ of {\em standard Schreier representatives} of words over $Y$ that generate $G$ can be defined as follows.
Fix a set $R$ of words over $Y$ whose images in $H$ form a set of right coset representatives for $G$. 
For all $r \in R$ and $y \in Y$, let $g_{r,y}$ be the word in $S$ given by $g_{r,y} = ryq^{-1}$,
where $q \in R$ represents the coset $G \overline{ry}$. 
Then the set $W = \{ g_{r,y} \; | \; r \in R, y \in Y \}$ generates $G$ (see, for example, p. 89 of \cite{\MKS}).

The Schreier rewriting process for $G$ with respect to these generators can be described as follows.
Consider the Schreier graph for $G$ in $H$ (in which vertices are labeled with the coset representatives from $R$ and edges are labeled with generators from $Y$).
For a given word $w \in S$, initialize $\tau(w)$ to $\lambda$. Trace $w$ through the Schreier graph. When traversing an edge from $r$ labeled $y$, update $\tau(w)$ to be $\tau(w)g_{r,y}$. 
Then $\tau$ is a rewriting process for $H$ with respect to the Schreier generators (see, for example, p. 91 of \cite{\MKS}). Since our sets $R$ and $W$ and the Schreier graph can all be stored in a finite amount of space, it is clear that $\tau$ can be computed in logspace.

\begin{prop}\label{prop:finiteExtQuot}
Let $G,H$ be finitely generated groups with $G$ a finite index
subgroup of $H$. Then $H$ has logspace normal form if and only if $G$
has logspace normal form.
\end{prop}
\begin{proof}
Throughout this proof we use the notation established above. 
We begin by assuming that $H$ has logspace normal form $h$. 
We define our normal form $g$ for $G$ as follows.
Each word $w$ over $X$ can be transformed into a word $w'$ over $Y$
by replacing the each occurrence of a letter $x_i$ with the corresponding word $w_i$
from $W$, and by replacing each occurrence of a letter $x_i^{-1}$ with the formal
inverse of the word $w_i$. 
Then $g(w)$ can be defined to be $\tau(h(w'))$.
Since $h$ and $\tau$ can both be computed in logspace,
by Lemma \ref{lem:comp}, so can $g$.

Next we assume that $G$ has logspace normal form $g$.
For a word $w$ over $Y$, we define $h(w)$ to be $g(w')r$, where
$r$ is the word in $R$ representing the coset $G \overline{w}$
and $w' = w r^{-1}$. 
To compute $h(w)$ in logspace, we trace $w$ in the Schreier graph to compute and store $r$.
We then call the normal form function $g$. When it asks for the $i$th
letter we supply it with the $i$th letter of $w r^{-1}$. 
When it asks to output a letter, we do so.
Finally we output $r$.
\end{proof}

\section{Closure under  wreath product}\label{sec:wreath}

In this section we prove that the property of having a logspace normal form is closed under
restricted wreath products. Propositions \ref{prop:indGenSet} and \ref{prop:emptyWord} allow us to assume  from now on that generating sets contain only non-trivial elements, and if $f_X$ is a logspace normal form function over a generating set $X$ then
 $f_X(\lambda)=\lambda$.

\begin{defn}  Given an ordered alphabet $X$, let $\le_{SL}$ denote the short-lex ordering on $X^*$. \end{defn}

\begin{lem}
\label{lem:shortLex}
Let $G$ be a group with symmetric generating set $X$ and logspace normal form function $f_X$.  The short-lex order of the normal form of two words in $X^*$ is logspace computable.\end{lem}

\begin{proof} Let $(u,v)\in X^*\times X^*$ be given.  We need to decide whether
 \begin{itemize}
\item $f_X(u)=f_X(v)$,\item $f_X(u)<_{SL} f_X(v)$, or 
\item $f_X(v)<_{SL} f_X(u)$.
\end{itemize}
We first call $f_X$ on $u$, but rather than write any output, each time a letter would be written to the output tape, we increase a counter, stored in binary.  We then do the same for $v$ and compare the two counters, if $|f_X(u)|< |f_X(v)|$ or $|f_X(v)|< |f_X(u)|$ we are done.

If not, call $f_X$ on $u$ and $v$ simultaneously to obtain the first letter of each output. If the letters are the same, obtain the next letter.  As soon as we encounter an $i$ for which the $i$th letters do not agree,
we can deduce which word is greater in the short-lex ordering, and if not, we deduce that $f_X(u)=f_X(v)$.
\end{proof}

We now establish some notation that will be useful in defining our normal form for $G\wr H$.
Let $G=\langle X\rangle$ and $H=\langle Y\rangle$  be groups with logspace normal form functions $f_X$ and  $f_Y$ respectively.  We may assume that $X$ and $Y$ are disjoint, finite symmetric generating sets.

Let $w\in (X\sqcup Y)^*$.  We will use $X(w)$ to denote the word in $X^*$ obtained by deleting all letters not in $X$ from $w$, and similarly for $Y(w)$.   For $w=a_1a_2\dots a_n$, $a_j\in X\sqcup Y$, and $1\leq i\leq n$, we will let $ X(i,w)$ (or $ Y(i,w)$) denote the word $X(a_1a_2\dots a_i)$ (or $Y(a_1a_2\dots a_i)$). For convenience, set $X(0,w)=\lambda$ and $Y(0,w)=\lambda$. Note that $X(i,w)$ (and $Y(i,w)$) are computable in logspace: if $w=a_1\ldots a_n$, set $j=0$; while $j<i$, increment $j$ by 1 and if  $a_j\in X$, output $a_j$.

Define $V(w)=\{f_Y(Y(s,w)) \ | \ 0\leq s\leq n\}$.
 Then $V(w)$ is a finite set of strings of $Y^*$. Note that $\lambda=f_Y(Y(0,w))$ is the shortest element in $V(w)$.  Set $v_0=\lambda$. The next lemma tells us how to compute the next element of $V(w)$ in  shortlex order in logspace, assuming the word $w$ is written on the input tape.
First, we need a way to store a word in $V(w)$ without using too much space, so to store a word in $V(w)$ corresponding to the element represented by $Y(i,w)$, we merely store the value $i$. To recover the word $v_i$, we run $f_Y$ on the word $Y(i,w)$.

\begin{lem}
\label{lem:shortLexV}
Let $w=a_1\dots a_n\in  (X\sqcup Y)^*$ be written on an input tape, and $V(w)=\{f_Y(Y(s,w)) \ | \ 0\leq s\leq n\}$. There is a logspace function which, given an integer $p$ such that $Y(p,w)=v_i$, computes $q$ such that  $Y(q,w)=v_{i+1}$ where $v_{i+1}$ is the next largest word from $v_i$ in shortlex order,  or returns that $v_i$ is the largest word in $V(w)$.
\end{lem}

\begin{proof}
 Feed $(Y(p,w),Y(1,w))$ into the algorithm in Lemma \ref{lem:shortLex}, and if $$f_Y(Y(p,w))<_{SL}f_Y(Y(1,w)),$$ set $q=1$. So $q$ encodes a word from $V(w)$ that is larger in the shortlex ordering than $v_i$ encoded by $p$.

For each $2\leq j\leq n$, read $a_j$, and if $a_j\in Y$, feed $(Y(p,w),Y(j,w))$ into the algorithm in Lemma \ref{lem:shortLex}. If $Y(j,w)$ is larger than $Y(p,w)$, check to see if $q$ has been assigned a value.  If not, set $q=j$.  If $q$ already has a value, feed  $(Y(q,w),Y(j,w))$  into the algorithm in Lemma \ref{lem:shortLex}. If $Y(j,w)$ is shorter than $Y(q,w)$, set $q=j$. So $q$ encodes an element in $V(w)$ that is greater that $v_i$ and less  than the previous  $Y(q,w)$.

When  every $j$ up to $n$ has been checked, if $q$ has not been assigned a value, then $v_i$ is the largest word in $V(w)$. Otherwise $q$ encodes the next largest word $v_{i+1}=Y(q,w)$.
\end{proof}

\begin{defn}[Normal form for $G\wr H$]
Let $w=a_1\dots a_n\in (X\sqcup Y)^*$. Then $$f_{X\sqcup Y}(w)=
u_1^{v_1}u_2^{v_2}\dots u_k^{v_k}f_Y(Y(w)),$$ where  $v_i\in f_Y(Y^*)$ with $v_i<_{SL} v_{i+1}$ and $u_i\in f_X(X^*)$, $u_i\not=\lambda$.
(Note that by $u_i^{v_i}$ we mean $ v_i u_i f_Y(v_i \inv)$. Lemma \ref{lem:inv} says $f_Y(v_i \inv)$ can be computed in logspace if $f_Y$ can.)

The words $v_i$ correspond to elements of $H$ for which the factor of $\bigoplus_{h\in H} G_h$ is non-trivial, so the $v_i$s are a subset of $V(w)$.
The words $u_i$ correspond to the non-trivial element of $G$ at each position $v_i$ in $H$. The prefix of the normal form word does the job of moving to each position in $H$ and fixing the value of $G$ at that position. The shortlex ordering of $V(w)$ allows us to do this in a canonical way for any word representing an element of $G\wr H$.
The suffix $f_Y(Y(w))$ takes us from the identity of $H$ to the final position in $H$.\end{defn}

Since we know how to compute the $v_i$ in shortlex order from Lemma \ref{lem:shortLexV}, all we need now is to compute the $u_i$ at each position.

\begin{lem}[Algorithm to compute $u_i$]
\label{lem:computeu}
Let $w=a_1\dots a_n\in  (X\sqcup Y)^*$ be written on an input tape,  $V(w)=\{f_Y(Y(s,w)) \ | \ 0\leq s\leq n\}$, and $p$ an integer such that $Y(p,w)=v_i\in V(w)$. There is a logspace function that decides whether the element of $G$ in the factor corresponding to $v_i$ in $\bigoplus_{h\in H} G_h$ is non-trivial, and a logspace function that outputs the normal form $f_X$ of this element.\end{lem}
\begin{proof}
Define a function $g_{X\sqcup Y}:(X\sqcup Y)^*\rightarrow X^*$ which computes a word in $X^*$ equal to the element of $G$ in the factor corresponding to $v_i=Y(p,w)$ in $\bigoplus_{h\in H} G_h$ as follows.

\begin{enumerate}
\item  compute the length  $n$ of $w$ and store it in binary.
\item set a counter $l=0$.
\item while $l<n$:
\begin{itemize} 
\item  call the function in Lemma \ref{lem:shortLex}  on $Y(p,w)$ and $Y(l,w)$ to decide if they are equal or not. If they are equal, set a boolean variable $b$ to be true, and otherwise set it to false.
\item while $l<n$ and  the letter at position $l+1$ is  in $X$:
\begin{itemize} 
\item if $b$ is true, print the letter at position $l+1$ to the output tape
\item increment $l$ by 1.
\end{itemize}
\end{itemize}
\end{enumerate}

Since the function in  Lemma \ref{lem:shortLex}  is logspace then so is $g_{X\sqcup Y}$. The algorithm works by scanning the input word from left to right, and outputting only those letters from $X(w)$ that are in the factor corresponding to $v_i$ in $\bigoplus_{h\in H} G_h$.

Then $f_X\circ g_{X\sqcup Y}$ will output the normal form word in $X^*$ for the element of $X$ in the copy of $G$ corresponding to the element $v_i\in H$.
To decide if this element is trivial or not, run the above procedure and test whether the output is $\lambda$ or not.
\end{proof}

\begin{thm} \label{closureWreath}
The normal form  function $f_{X\sqcup Y}$ for $G\wr H$ can be computed in logspace.\end{thm}
\begin{proof}
Set $p=0$ (so  $Y(p,w)=v_0=\lambda$,  the shortest element in $V(w)$).  
Set a boolen variable max to be false.
While max is false:
\begin{itemize}
\item 
use Lemma \ref{lem:computeu} to determine whether the element in $G$ at $Y(p,w)$ is non-trivial. If it is, output $f_Y(Y(p,w))=v_i$.
Then output $u_i$ by running the algorithm in  Lemma \ref{lem:computeu} again. Then output $f_Y(v_i^{-1})$ (apply Lemma \ref{lem:inv} to function that computes $f_Y(Y(p,w))$).
\item run the algorithm Lemma  \ref{lem:shortLex} with input $p$. If the algorithm returns that $Y(p,w)$ is maximal in $V(w)$, set the variable max to be true. Otherwise it finds $q$ such that $Y(p,w)=v_i$ and $Y(q,w)=v_{i+1}$. Set $p=q$.
\end{itemize}
Finally, output $f_Y(Y(w))$.
 \end{proof}

It follows that the class of groups with logspace normal form includes the  so-called  {\em lamplighter groups}, and the group $\Z\wr \Z^2$ (which Parry considered in \cite{\Parry}, showing with respect to a standard generating set finding a geodesic form for a given word is $NP$-hard, and so a geodesic normal form for it is unlikely to be logspace computable).

In \cite{\Waack} Waack gives an example of a group with logspace word problem that is not residually finite, and hence non-linear.
Theorem \ref{closureWreath} allows us to construct non-linear and non-residually finite groups having logspace normal forms.

\begin{cor}
Not all groups with a logspace normal form are linear.
\end{cor}
\begin{proof}
By Corollary 15.1.5 in \cite{\Robinson}, $(\Z \wr \Z) \wr \Z$ is not
linear, but it has a logspace normal form by Theorem
\ref{closureWreath}.
\end{proof}

\begin{cor}
Not all groups with logspace normal forms are residually finite.
\end{cor}
\begin{proof}
Let $G$ be the wreath product of the symmetric group $S_3$ on three letters
and $\Z = \langle t \rangle$. By Theorem \ref{closureWreath}, $G$ has logspace normal form.
But it is easy to show that $G$ is not residually finite.
Let $\theta$ be a homomorphism from $G$ to a finite group.
We will show that $\theta$ kills the commutator subgroup $[S_3, S_3]$.
Let $n$ be a positive integer such that $t^n \theta = 1$.
Since $[S_3, S_3^{t^n}] = 1$,
$$[S_3, S_3]\theta = [S_3 \theta, (S_3^{t^n}) \theta] = [S_3, S_3^{t^n}] \theta = 1.$$
\end{proof}

\section{Closure under free products}\label{sec:freeprod}
Unfortunately we are not able to prove closure of logspace normal
forms under free product in general, but we are able to do so  if the free product has
logspace word problem, for example if it is a free product of linear
groups. 

Let $G = \langle X \rangle$ and $H = \langle Y \rangle$ be groups with
logspace normal form functions $g_X$ and $h_Y$, and suppose $X$ and $Y$ are disjoint.  By Proposition \ref{prop:emptyWord} we can assume that $g_X$ and $h_Y$ both
have the property that the normal form for a word representing the
identity is $\lambda$.
We will define a normal form function for the free product $G \ast H$, which is generated by $X \sqcup
Y$.

We start with the following lemma.
\begin{lem}
\label{inG}
Let $w = u_1 v_1 u_2 v_2 \dots u_k v_k$ where $u_i \in X^*$ and $v_i
\in Y^*$.  Then $w$ represents an element in $G$ if and only if $w
u_k^{-1} u_{k-1}^{-1} \dots u_1^{-1}=_{G\ast H} 1.$
\end{lem}
\begin{proof}
We proceed by induction on $k$. For $k=1$ we have $u_1v_1\in G$
meaning $v_1\in G$, but since $v_1\in Y^*$ with $Y$ disjoint from $X$,
we must have $v_1=1$, so $u_1v_1u_1^{-1}=1$.

Assume the result is true for $k$, and let $w=u_1v_1\dots
u_{k+1}v_{k+1}$ represent an element in $G$. By the Normal Form
Theorem for free products (\cite{\LS}, p. 175), $w$ has a unique reduced
form consisting of a single subword $u\in X^*$, so we must have
$u_i=1$ for some $i>1$ or $v_i=1$ for some $i<k+1$. If $u_i=1$ then $$w=u_1v_1\dots
u_{i-1}(v_{i-1}v_i)u_{i+1}\dots v_{k+1}$$ and by the induction hypothesis
$$wu_{k+1}^{-1}u_k^{-1} \dots u_{i+1}^{-1}u_{i-1}^{-1}\dots u_1^{-1}=1.$$ Similarly if
$v_i=1$.

The converse is clearly true.
\end{proof}

We define a normal form for $G\ast H$ recursively as follows.
\begin{defn}[Normal form for $G\ast H$]
  Let $w \in (X \sqcup Y)^*$.
\begin{enumerate}
\item write $w$ as  a freely reduced word.
\item if $\overline w\in G$, then define  $f_{X\sqcup Y}(w)=g_X(X(w))$.
\item if $\overline w\in H$, then define  $f_{X\sqcup Y}(w)=h_Y(Y(w))$.
\item otherwise, let $w_1$ be the longest initial segment of $w$ such that $\overline{w_1} \in G$,
and let $w'$ be the tail of $w$, so $w = w_1 w'$.
Then  \begin{itemize}
\item if $w_1$ has nonzero length, define
$$f_{X \sqcup Y}(w) = g_X(X(w_1)) f_{X \sqcup Y}(w').$$
\item otherwise, let $w_2$ be the longest  initial segment of $w$ such that $\overline{w_2} \in H$,
and let $w'$ be the tail of $w$, so $w = w_2 w'$.
Note that if $w_2=\lambda$ then the first case applies, so $w_2$ has nonzero length.

In this case, define
$$f_{X \sqcup Y}(w) = h_Y(Y(w_2)) f_{X \sqcup Y}(w').$$
\end{itemize}\end{enumerate}
\end{defn}

\begin{prop}
The normal form function $f_{X \sqcup Y}$ is well defined.
\end{prop}
\begin{proof}
In case (1), by the previous lemma we have $w=_{G\ast H}X(w)$ (the word obtained from $w$ by deleting all letters from $Y$). So the normal form function $g_X$ applies and gives a unique representative for $w$. Similarly for case (2). So words that lie completely in one of the factors have a well-defined normal form.

Now let   $w\in (X\sqcup Y)^*$ be freely reduced, and assume $\overline{w}\not\in G$ and  $\overline{w}\not\in H$.  Then $\overline{w}$ is
non-trivial (since it is not in $G$ or $H$). 
For each non-identity element of $G \ast H$, there is a unique way to represent it as
an alternating product of non-identity elements in $G$ and $H$ \cite{\LS}.
So write $w =_{G\ast H} u_1 v_1 u_2 v_2 \dots u_k v_k$ where $u_i \in X^*$ and $v_i
\in X^*$. Since $\overline{w}$ does not lie in $G$ or $H$, it has at least two factors.

If the alternating product starts with $u_1\in X^*$, then we claim that any word representing $\overline w$ has a longest initial segment that evaluates to an element of $G$, and this element is equal to $u_1$. If so, then the choice made by $f_{X\sqcup Y}$ is unique.
Take the (freely reduced) word $w$ and write it as $a_1b_1\dots a_lb_l$ with $a_i\in X^*$ and $b_i\in Y^*$ with only $a_1, b_l$ allowed to be empty words.
Then $v_k^{-1}u_k^{-1} \dots v_1^{-1} u_1^{-1} a_1b_1\dots a_lb_l =_{G\ast H} 1$, so by the normal form theorem  \cite{\LS} some term must represent the identity in $G$ or $H$, and this term must involve $u_1$. So $w$ has a prefix which is equal to $u_1$, and since
$\overline w\not\in G$, $v_1$ is not empty, so there is a longest prefix of $w$ that equals $u_1$, and cancels so that $v_1$ can then cancel.

A similar argument applies if the alternating product starts with $v_1u_2$.
\end{proof}

In Proposition \ref{prop:freegp} we proved that free groups of finite rank have logspace normal form, using the fact that they have  logspace word problem. We generalise this argument to show that the function $f_{X\sqcup Y}$ can be computed in logspace, provided the word problem for $G\ast H$ can be decided in logspace.

\begin{prop}
\label{prop:freeClosure}
Let $G = \langle X \rangle$ and $H = \langle Y \rangle$ be groups with
logspace normal forms.  Suppose furthermore that $G \ast H$ has
logspace decidable word problem.  Then $G \ast H$ has logspace normal
form.
\end{prop}
\begin{proof}
Let $w \in (X \sqcup Y)^*$ be the freely reduced word equal to the input word (run the logspace algorithm in Proposition \ref{prop:freegp} on the input word to obtain it).
By Lemma \ref{inG} we can compute $j$ such that $w_1 = u_1 v_1 u_2 v_2
\dots v_{j-1} u_j$ is the longest initial segment of $w$ such that
$\overline{w_1} \in G$, by inputting $u_1 v_1 u_2 v_2
\dots v_{j-1} u_{j-1}^{-1}\dots u_1^{-1}$ into the logspace word problem function for $G\ast H$.

  Output the normal form $g_X$ for $u_1 u_2
\dots u_j$ and move the input pointer to point to $v_j$.  For ease of
notation, we rename $v_j u_{j+1} v_{j+1} \dots u_k v_k$ to be our new
$w$, and reindex so that $w = v_1 u_1 v_2 u_2 \dots v_r u_r$.
Compute $j$ such that $w_1 = v_1 u_1 v_2 u_2 \dots u_{j-1} v_j$ is
the largest initial segment of our new $w$ such that $\overline{w_1}
\in H$.  Output the normal form over $Y$ for $v_1 u_1 v_2 u_2 \dots
u_{j-1} v_j$ and move the input pointer to point to $u_j$.  Continue
in this way until the entire input has been processed.  Note that at
any one stage we are only storing a constant number of pointers to the
input.
\end{proof}

\begin{cor}
\label{cor:freeProdLinear}
Let $F$ be a field, and let $G$ and $H$ be linear over $F$ with a
logspace normal form.  Then $G \ast H$ is also linear over $F$ and it
also has a logspace normal form.
\end{cor}
\begin{proof}
The class of groups which are linear over $F$ is closed under free
products (see, for example, Corollary 2.14 in \cite{\Wehrfritz}).  All
linear groups groups have logspace word problem (for the case when the
characteristic of $F$ is $0$, see \cite{\LiptonZ}; for the positive
characteristic case see \cite{\Simon}).  The corollary follows.
\end{proof}

\section{Closure under infinite extensions}
\label{sec:infiniteextensions}

In this section we prove that certain infinite extensions
of groups with logspace normal forms also have logspace normal form.
 If $N$ is a normal subgroup of finitely generated group $G = \langle X \rangle$,
if $G/N$ has logspace normal form, and if there is a logspace computable
function to produce normal forms for elements of $N$ in terms of generators
in $X$, then $G$ has logspace normal form.
This will enable us to extend our class to include certain amalgamated products
and one-relator groups.

Notice that for the following lemma, $N$ need not be finitely generated,
and we posit the existence of  a function similar to a normal form function for $N$
in the sense that it produces unique representatives for the elements of $N$,
but different in the sense that it is defined on words over the generators for
the ambient group $G$.

\begin{lem}
\label{lem:someExt}
Let $G$ be a group with normal subgroup $N$.
Let $X$ be a finite symmetric generating set for $G$.
Let $S=\{w\in X^* \ | \ \overline{w} \in N\}$.  
Suppose that
\begin{itemize}
\item $G/N$ has logspace normal form; and
\item there is a logspace computable function $f: S \rightarrow S$
such that $\overline{f(w)}= \overline{w}$ and $f(w_1) = f(w_2)$ if and only if $\overline{w_1} = \overline{w_2}$.
\end{itemize}
Then $G$ has logspace normal form.
\end{lem}
\begin{proof}
Let $X_N=\{xN \ | \ x\in X\}$ be a generating set for $G/N$,  and let $h:(X_N)^*\rightarrow (X_N)^*$ be the logspace normal form function for $G/N$ with respect to this generating set. 
Define two logspace functions $p:X^*\rightarrow (X_N)^*$ and $q:(X_N)^*\rightarrow X^*$ by $p(x)=xN$ for each $x\in X$, and $q(xN)=x$ for each $xN\in X_N$.
Let $\iota:X^*\rightarrow X^*$ be the logspace function  $\iota(a_1\dots a_n)=a_n^{-1}\dots a_1^{-1}$ that computes the inverse of a word.


Define a function $g:X^*\rightarrow X^*$ as follows:
 On input 
 $w\in X^*$:
 \begin{enumerate}
 \item compute $q(h(p(w)))$ writing the output word $b_1\dots b_k$ to the output tape, storing the integer $k$
 \item call the function $f$, and when it asks for the $i$th input letter:
 \begin{itemize}
  \item if $i\leq k$, call $\iota(q(h(p(w))))$ and return the $i$th letter of its output;
  \item if $i>k$, return the $(i-k)$th letter of $w$.
\end{itemize}\end{enumerate}
 
 Since $w\in wN$, step (1) of the algorithm computes $b=b_1\dots b_k$ such that $wN=(b_1N)\dots (b_kN)$ and returns the letters $b_1,\dots, b_k$. Then $w=bn$ for some $n\in N$, and since $n=b^{-1}w$, step (2) of the algorithm outputs $f(b^{-1}w)$.
\end{proof}

We first explore some implications of Lemma \ref{lem:someExt} for amalgamated products. We are grateful to Chuck Miller for his invaluable input into the remainder of this section.

\begin{cor}
Suppose that $G$ has logspace normal form,
and that $N$ is a normal subgroup such that $G/N$ is linear and has logspace normal form.
Then the amalgamated product $H$ of $G$ with itself along $N$
has logspace normal form.
\end{cor}
\begin{proof}
Fix a finite generating set $X$ for $G$, and define two new copies $X_1$ and $X_2$ of $X$ as follows:
for each element $x_j\in X$  make two new generators $(x_j)_1$ and $(x_j)_2$, and for $i=1,2$, 
let $X_i = \{(x_j)_i \; | \; x_j \in X \}$.
Let $G_1$ and $G_2$ be two copies of $G$ with generators $X_1$ and $X_2$ respectively,
and let $N_1$ and $N_2$ be corresponding normal subgroups of $G_1$ and $G_2$.


We have $H/N
 = G_1/N_1 \ast G_2/N_2$ where $G_i/N_i$ is linear and has logspace normal form. Hence by Corollary \ref{cor:freeProdLinear}, $H/N$
has logspace normal form, with respect to the generating set $X_1\sqcup X_2$.

Let $S$ be the set of words $w$ over $X_1 \sqcup X_2$ such that $\overline{w} \in N$. 
Write $w\in S$ as an alternating product of subwords from $X_1$ and $X_2$.
 Since $\overline{w} \in N$, $w$ is equal in $H$ to a 
 word $u\in X_1^*$ with $\overline{u}\in N_1$. Then $wu^{-1}=1$ in $H$.
 Write $wu^{-1}$ as an alternating product $u_1v_1\dots u_kv_k$ with $u_i\in X_1^*, v_i\in X_2^*$ (with all subwords nonempty except possibly $u_1$ and $v_k$).
By the normal form theorem for amalgamated free products  \cite{\LS}, the alternating word  contains a subword  $u_i\in X_1^*$  with  $\overline{u_i} \in N$, or  
$v_i\in X_2^*$  with  $\overline{v_i} \in N$. In the first case write $u_i$ as a word in $X_2^*$ by replacing each $(x_j)_1$ letter by $(x_j)_2$, and $v_i$ as a word in $X_1^*$ by replacing each  $(x_j)_2$ letter by $(x_j)_1$ in the second case. The resulting word is also equal to $1$ in $H$, so if it contains letters from both generating sets, another subword can rewritten, reducing the number of alternating subwords, so that after a finite number of iterations the word $wu^{-1}$ is equal in $H$ to a word obtained by replacing all letters $(x_j)_2$ by $(x_j)_1$.

It follows that $w$ is equal in $H$ to the word obtained from $w$ by replacing each $(x_j)_2$ letter by $(x_j)_1$. Let $p:(X_1 \sqcup X_2)^*\rightarrow (X_1 \sqcup X_2)^*$ be the map that performs this substitution, so clearly $p$ can be computed in logspace, and $\overline {p(w)}=\overline {w}$ for all $w\in S$. Let $g_{X_1}$ be the logspace normal form function for $G_1$ (since $G$ has logspace normal form it follows that $G_1$ does). Then $f=g_{X_1}\circ p$ is logspace computable  by Lemma \ref{lem:comp}.

Since $\overline{p(w)}$ is equal to $\overline{w}$ in $H$, and $g_{X_1}$ is a normal form function,  we have  $\overline{f(w)}= \overline{w}$.
If  $f(u) = f(v)$  then $g_{X_1}(p(u))=g_{X_1}(p(v))$,  so 
$\overline{p(u)}=\overline {p(v)}$ since $g_{X_1}$ is a normal form function, which implies $\overline{u}=\overline {v}$. Since we have satisfied the criteria of  Lemma \ref{lem:someExt}, the result follows.
\end{proof}

\begin{cor}
Let $F$ be the free group on two generators.
Then the amalgamated product of $F$ with itself along the commutator subgroup
has logspace normal form.
\end{cor}
\begin{proof}
The abelianization of $F$ is a free abelian group, and hence is linear and has logspace normal form.
\end{proof}

\begin{cor}
Let $\mathrm{BS}(1,p)$ be a Baumlag-Solitar group (as defined in Section \ref{sec:BS}).
Then the amalgamated product of $\mathrm{BS}(1,p)$ with itself along the commutator subgroup
has logspace normal form.
\end{cor}
\begin{proof}
The abelianization of $\mathrm{BS}(1,p)$ is cyclic, and hence is linear  and has logspace normal form.
\end{proof}

We next explore some implications of Lemma \ref{lem:someExt} for 
torus knot groups.

\begin{lem}
\label{lem:oneRel}
Let $G=\langle a, b \; | \; a^m = b^n \rangle$, where $m$ and $n$ are positive integers, and let $N$ be the subgroup of $G$ generated by $a^m$.
Let $$w = a^{r_1}b^{s_1}a^{r_2}b^{s_2}\dots a^{r_k}b^{s_k}$$ such that $\overline{w} \in N$.
Then
$m$ divides $\sum_{j=1}^k r_j$, $n$ divides $\sum_{j=1}^k s_j$, and
$\overline{w} = \overline{a^{mi}}$, where $$i = \frac1{m}\sum_{j=1}^k r_j + \frac1{n} \sum_{j=1}^k s_j.$$
\end{lem}
\begin{proof}
We proceed by induction on $k$.
The case when $k=1$ is clear.
Assume that $k>1$ and that our result holds for $k-1$. Then by the Normal Form Theorem for Free Products with Amalgamation (\cite{\LS}, Theorem 2.6) either $m$ divides $r_i$ for some $i$,
or $n$ divides $s_i$ for some $i$.
Let us assume the latter, the argument being the same in either case.
Since $\overline{b^{s_i}}$ is central, we may assume that
$$w = a^{r_1}b^{s_1}a^{r_2}b^{s_2} \dots b^{s_{i-1}}a^{r_i + r_{i+1}}b^{s_{i+1}} \dots a^{r_k}b^{s_k + s_i}.$$
Our result now follows from our inductive assumption.
\end{proof}

\begin{cor}
The torus knot group $\displaystyle G= \langle a, b \; | \; a^m = b^n \rangle$ for $m,n$ positive integers has logspace normal form.
\end{cor}
\begin{proof}
Let $N$ be the subgroup of $G$ generated by $a^m$, which is normal since
$N$ is central.
$G/N$ is the free product of two finite cyclic groups, so by
Corollary \ref{cor:freeProdLinear}, it has a logspace normal form function.
Let $S$ be the set of words representing elements of $N$.
Let $f:S \rightarrow S$ be the function that takes a word in $S$ to its representative of the form $a^{mi}$.
By Lemma \ref{lem:oneRel} we can calculate $f$ using two counters. This can be done in logspace. The result then follows from Lemma \ref{lem:someExt}.
\end{proof}

Note that the braid group on three strands has presentation $ \langle a, b \; | \; a^2 = b^3 \rangle$. Since braid groups on $n$ strands are linear, it would be interesting to know whether or not they admit logspace normal forms for $n>3$. We thank Arkadius Kalka for pointing this out.

\section{Solvable Baumslag-Solitar groups}\label{sec:BS}

Let  $G = \langle a, t \; | \; t a t^{-1} = a^p \rangle$ for $p\geq 2$, and $X=\{a^{\pm 1}, t^{\pm 1}\}$.
Note that $G$ is isomorphic to the set of all matrices of the form
\begin{displaymath}
\left( \begin{array}{cc} p^i & m \\ 0 & 1
\end{array} \right),
\end{displaymath}
where $i \in \Z$ and $m \in \Z[\frac1{p}]$, where the isomorphism is
given by
\begin{displaymath}
t \rightarrow \left( \begin{array}{cc} p & 0 \\ 0 & 1
\end{array} \right), \ 
a \rightarrow \left( \begin{array}{cc} 1 & 1 \\ 0 & 1
\end{array} \right).
\end{displaymath}

We obtain a normal form as follows.  Write \begin{displaymath}
\left( \begin{array}{cc} p^i & m \\ 0 & 1
\end{array} \right) =\left( \begin{array}{cc} 1 & m \\ 0 & 1
\end{array} \right) \left( \begin{array}{cc} p^i & 0 \\ 0 & 1
\end{array} \right).
\end{displaymath}
Then $m\in\mathbb Z[\frac1{p}]$  has a unique $p$--ary expansion as
either \begin{itemize}
\item $0$,
\item $\frac{\eta_0}{p^{\alpha_0}}+\frac{\eta_1}{p^{\alpha_1}}+\dots+\frac{\eta_k}{p^{\alpha_k}}$
with $0<\eta_j<p$ and $\alpha_0>\alpha_1>\dots >\alpha_k$, or
\item $\frac{\eta_0}{p^{\alpha_0}}+\frac{\eta_1}{p^{\alpha_1}}+\dots+\frac{\eta_k}{p^{\alpha_k}}$
with $-p<\eta_j<0$ and $\alpha_0>\alpha_1>\dots >\alpha_k$
\end{itemize}
where the $p$--ary
expansion for $m$ is written from least to most significant bits.
Finally note that
 \begin{displaymath}
\left( \begin{array}{cc} 1 & \frac{\eta_j}{p^{\alpha_j}} \\ 0 & 1
\end{array} \right) = \left( \begin{array}{cc} \frac1{p^{\alpha_j}} & 0 \\ 0 & 1
\end{array} \right) \left( \begin{array}{cc} 1 & \eta_j \\ 0 & 1
\end{array} \right) \left( \begin{array}{cc} p^{\alpha_j} & 0 \\ 0 & 1
\end{array} \right)=t^{-\alpha_j}a^{\eta_j}t^{\alpha_j},
\end{displaymath}
so
it follows that each element
of $G$ can be written uniquely in one of the following
three forms:
 \begin{itemize}
\item $t^i$,
\item $(a^{\eta_0})^{t^{\alpha_0}} (a^{\eta_1})^{t^{\alpha_1}} \dots (a^{\eta_k})^{t^{\alpha_k}} t^i $, 
\item $ (a^{-\eta_0})^{t^{\alpha_0}} (a^{-\eta_1})^{t^{\alpha_1}} \dots (a^{-\eta_k})^{t^{\alpha_k}} t^i $,
\end{itemize}
where $i,k \in \Z$, $k\geq 0$, $0< \eta_j <p$,   $\alpha_0 > \alpha_1> \dots > \alpha_k$, and $x^y=y^{-1}xy$.


For example, for $p=2$,
\begin{displaymath}
\left( \begin{array}{cc} 8 & \frac{11}{4} \\ 0 & 1
\end{array} \right)
\end{displaymath}
can be written as \begin{displaymath}
\left( \begin{array}{cc} 1 & \frac{1}{2^2} \\ 0 & 1
\end{array} \right)\left( \begin{array}{cc} 1 & \frac{1}{2^1} \\ 0 & 1
\end{array} \right)\left( \begin{array}{cc} 1 & 2^1 \\ 0 & 1
\end{array} \right) \left( \begin{array}{cc} 2^3 & 0\\ 0 & 1
\end{array} \right)= a^{t^{2}} a^{t^{1}}a^{t^{-1}}t^3.\end{displaymath}

Define the {\em level} of a letter in a word $w\in X^*$ to be the $t$-exponent sum of the prefix  of $w$ ending with this letter.
For example,  the levels of the $a$ letters in the word $a^{t^{2}} a^{t}a^{t^{-1}}t^3$ are $-2,-1,1$ respectively. 

 If $w\in X^*$, let $\mathrm{texp}$ denote the $t$-exponent sum of $w$, $l_{\min}$ the minimum level of any letter in $w$, and  $l_{\max}$ the maximum level of a letter in $w$.
\begin{lem}\label{lem:step1}
 If $w\in X^*$  is  written on an input tape, then we can compute and store  $\mathrm{texp}, l_{\min}$ and $l_{\max}$ in logspace. \end{lem}
  \begin{proof}
  We perform the following logspace algorithm:
\begin{enumerate}\item set binary counters  $\mathrm{texp},l_{\min}, l_{\max}$ to zero
\item scan the input from left to right
\begin{itemize}
\item if the next letter is $t$, increment $\mathrm{texp}$ by 1, and set \\ $l_{\max}= \max\{\mathrm{texp},l_{\max}\}$
\item if the next letter is $t^{-1}$, decrement $\mathrm{texp}$ by 1,  and set \\ $l_{\min}= \min\{\mathrm{texp},l_{\min}\}$
\item if the next letter is $a^{\pm 1}$, do nothing.
\end{itemize}\end{enumerate}
When the end of the input is reached, the counters  $\mathrm{texp},l_{\min}, l_{\max}$ contain the required values for $w$, and have absolute value no more than the length of the input. 
\end{proof}

We will use the fact  that since $G$ is metabelian, words of zero $t$-exponent sum  commute in $G$. 
For example, if $u=atatatat^{-2}at^{-2}at$, the subword $tatat^{-2}$   has zero $t$-exponent sum, and so we may commute it past the first $a$ at level 1 to obtain $\displaystyle at(tatat^{-2})aat^{-2}at$, as illustrated in Figure \ref{fig1}.  This means we may collect together $a^{\pm 1}$ letters at the same level without changing the group element represented by a word.
\begin{centering}
\begin{figure}[h!]
\begin{tikzpicture}[scale=.8]
   \node (la) at (0,5) {level 3};
   \node (lb) at (0,4) {level 2};
   \node (lc) at (0,3) {level 1};
      \node (ld) at (0,2) {level 0};
         \node (le) at (0,1) {level -1};
         
         \draw[dotted] (1,3)--(7,3);
         
         \draw[-latex]    (1,2)--(2,2); 

         \draw[-latex]    (2,2)--(2,3);       
               \node (a) at (1.7,2.5) {$t$};
               
               \draw[-latex]    (2,3)--(3,3); 
               
                   \node (a) at (2.5,2.7) {$a$};

       \draw[-latex]    (3,3)--(3,4); 
              \draw[-latex]    (3,4)--(4,4); 
              \draw[-latex]    (4,4)--(4,5);       
                     \draw[-latex]    (4,5)--(5,5);    
              \draw[latex-]    (5,5)--(5,4);    
              \draw[latex-]    (5,4)--(5,3);    
        \draw[-latex]    (5,3)--(6,3);    
        
              \node (a) at (5.5,2.7) {$a$};
              
        \draw[latex-]    (6,3)--(6,2);    
        \draw[latex-]    (6,2)--(6,1);    
        \draw[-latex]    (6,1)--(7,1);
                \draw[-latex]    (7,1)--(7,2);

                         \draw[dotted] (8,3)--(14,3);
         
         \draw[-latex]    (8,2)--(9,2); 

         \draw[-latex]    (9,2)--(9,3);       
    
                 \node (a) at (8.7,2.5) {$t$};

       \draw[-latex]    (9,3)--(9,4); 
              \draw[-latex]    (9,4)--(10,4); 
              \draw[-latex]    (10,4)--(10,5);       
                     \draw[-latex]    (10,5)--(11,5);    
              \draw[latex-]    (11,5)--(11,4);    
              \draw[latex-]    (11,4)--(11,3);    
        \draw[-latex]    (11,3)--(12,3);    
        
              \node (a) at (11.5,2.7) {$a$};
                   \draw[-latex]    (12,3)--(13,3); 
                   
                         \node (a) at (12.5,2.7) {$a$};
                   
        \draw[latex-]    (13,3)--(13,2);    
        \draw[latex-]    (13,2)--(13,1);    
        \draw[-latex]    (13,1)--(14,1);
                \draw[-latex]    (14,1)--(14,2);

   \end{tikzpicture}
 \caption{Subwords of zero $t$-exponent commute with $a$ letters.}

   \label{fig1}
\end{figure}
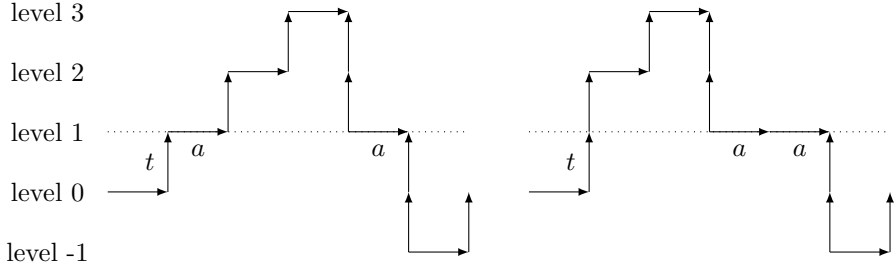
\end{centering}

Let $S\subset X^*$ be the set of words of zero $t$-exponent sum. We define a function    $f_S:S\rightarrow S$ such that $f_S(u)$ is 
 a word  of the form $\lambda$ or 
$$(a^{\eta_0})^{t^{\alpha_0}} (a^{\eta_1})^{t^{\alpha_1}} \dots  (a^{\eta_{k-1}})^{t^{\alpha_{k-1}}} (a^{\beta})^{t^{\alpha_k}},$$ where   $\alpha_0 > \alpha_1> \dots > \alpha_k$, $0<\eta_i<p$, $k\geq 0$, and $\beta\in\mathbb Z, \beta\neq 0$, such that $u$ and $f_S(u)$ represent the same element in $\mathrm{BS}(1,p)$. 
Note that the level of each $a^{\pm 1}$ in the subword $(a^{\eta_i})^{t^{\alpha_i}}$ is $-\alpha_i$. Note also  
  that since $\beta$ is allowed to range through all integers,
 the output here is not guaranteed to produce a unique representative 
for each group element, so we do not claim $f_S$ is a normal form function. Instead,  we call $f_S(u)$  an  {\em approximation} of the normal form for $u$. It is 
 computed using the following algorithm.

\begin{alg}[Approximation Algorithm]\label{alg:approx}
 On input $u\in S$:  
{\em\begin{enumerate}
\item run the algorithm in Lemma \ref{lem:step1} and store $l_{\min} $ and $ l_{\max}$ in binary.
	\item set $l=l_{\min}, \mathrm{aexp}=0$ and $\beta=0$. 
	\item while $l \leq l_{\max}$:
	\begin{enumerate}
		\item set texp = 0
	\item \label{item:accum} for each letter $x$ of input, starting at the left most letter: 

	\begin{itemize}

			\item if $x=t$, put $\mathrm{texp}= \mathrm{texp}+1$
			\item if $x=t^{-1}$,  put $\mathrm{texp}= \mathrm{texp}-1$

			\item if  $x=a$ and $\mathrm{texp}=l$, put $\mathrm{aexp}= \mathrm{aexp}+1$
		 	\item   if  $x=a^{-1}$ and $\mathrm{texp}=l$, put $\mathrm{aexp}= \mathrm{aexp}-1$
		\end{itemize}

		\item if $l<l_{\max}$: 

			\begin{itemize}		
				\item compute and store (in binary) $q,r\in\mathbb Z$ \\ where   $\mathrm{aexp}=pq+r$ and $0\leq r<p$ 						
				
				\item if $r\neq 0$ write $(a^r)^{t^{-l}}$ to the output tape, and set $\beta=r$
				\item   \label{item:carry}  set $\mathrm{aexp} = q$
			\end{itemize}
		\item[] \label{item:beta} else:
		
			\begin{itemize}	\item if $\mathrm{aexp}\neq 0$:
				\begin{itemize}	
				\item set $\beta=\mathrm{aexp}$ 
		
				\item  write $(a^{\beta})^{t^{-l_{\max}}}$ to the output tape
				\end{itemize}
	
 				\item	set $\mathrm{aexp} = 0$
				\item increment $l$ by 1. 
			\end{itemize}
		\item\label{item:loopinv} loop invariant: if $v$ is word written on the output tape then we claim that 
		\begin{displaymath} 
		\overline{ u} = \overline{v (a^{\mathrm{aexp}})^{t^{-l}}[u]_l},
		\end{displaymath}
		where $[u]_l$ is the word obtained from $u$ by ignoring all those occurrences
		of $a$ and $a^{-1}$ at levels  less than $l$.    
	\end{enumerate}

\end{enumerate}}
\end{alg}

Notice that we are effectively computing the $p$--ary expansion of the upper-right entry 
in our matrix representation, and that to do so requires that  ``carrying" $a$-exponent sums 
from one level to the next; this is the job of the variable $\mathrm{aexp}$ in the algorithm.

Note also that the variable $\beta$ is changed  in step 3(c) only if the new value is nonzero and some  $a^{\pm 1}$ letters are written to the output tape. 
So at the end of the algorithm $\beta=0$ if and only if the output word has no  $a^{\pm 1}$ letters, and the algorithm has no output.
If $\beta\neq 0$ then the suffix of the output word is $(a^{\beta})^{t^{-l}}$.



The next three lemmas show that Algorithm \ref{alg:approx} computes the function $f_S$ as required, and in logspace.

\begin{lem}
\label{lem:canCommute}
Let $u$ be a word in $X^*$ with $t$-exponent sum equal to $zero$. 
Let $m$ be the minimum level of any $a$ or $a^{-1}$ in $u$.
Let $e$ be the exponent sum of those $a$s in $u$ which are at level $m$. 
Then $\overline{u} = \overline{(a^{e})^{t^{-m}} [u]_{m+1}}.$
\end{lem}
\begin{proof}
Let $a_1,a_2,\dots ,a_s$ be the $a^{\pm 1}$ letters in $u$ that are at level $m$.
Write $u$ as $u_0a_1u_1a_2u_2\dots a_su_s$ where all $a^{\pm 1}$ letters in $u_i$ are above level $m$, so $u_0u_1\ldots u_s=[u]_{m+1}$.
Then the $t$-exponent sum of $u_0$ is $m$, the $t$-exponent sum of $u_i$ for $1\leq i<s$ is zero, and for $u_s$ is $-m$.
Inserting $t^{-m}t^m$ pairs before and after each $a_i$ we obtain a word 
$$v=u_0t^{-m}(t^ma_1t^{-m})t^mu_1t^{-m}(t^ma_2t^{-m})t^mu_2\dots t^{-m}(t^ma_st^{-m})t^mu_s$$ with $\overline u=\overline v$. Put $v_0=u_0t^{-m}, v_i=t^mu_it^{-m}$ for $1\leq i<s$ and $v_s=t^mu_s$, so each $v_i$ has zero $t$-exponent sum, and $v=v_0t^ma_1t^{-m}v_1t^ma_2t^{-m}v_2\dots t^ma_st^{-m}v_s$.
Then $$\overline v=\overline {t^ma_1a_2\dots a_st^{-m}   v_0v_1\dots v_s}$$ since words of zero $t$-exponent sum commute, where
$\overline{a_1a_2\dots a_s}=\overline{a^e}$. Finally note that   $v_0v_1\dots v_s=u_0t^{-m}t^mu_it^{-m}\dots t^mu_{s-1}t^{-m}t^mu_s$ which after cancellation of $t^{-m}t^m$ pairs is $[u]_{m+1}$.
\end{proof}

\begin{lem}
The approximation algorithm is correct.
\end{lem}
\begin{proof}
It is clear that the format of the output word is correct, so it remains to show that the 
output word is equal in the group to the input word. 
We will first show that the loop invariant defined in step \ref{item:loopinv} is preserved under one iteration of the main loop. 
Let $l$, $\mathrm{aexp}$ and $v$ represent the values of these variables at the start of an iteration,
that is, at the top of the loop, where $v$ is the word currently written on the output tape. Assume the loop invariant holds, so
 $\overline u=\overline {v(a^{\mathrm{aexp}})^{t^{-l}}[u]_l}$.

Let $l'$, $\mathrm{aexp}'$ and $v'$ represent the values of these variables at the end of 
that iteration, that is, at the point of the loop which is marked with the loop invariant. Note that $l'=l+1$.

Let $e$ be the exponent sum of those $a^{\pm 1}$s at level $l$,
the lowest level of any $a^{\pm 1}$ letters in $[u]_l$.
The algorithm sets  $\mathrm{aexp}' = q$ where $\mathrm{aexp} + e = pq + r$
and writes $(a^r)^{t^{-l}}$ to the output tape if $r>0$.
So $\displaystyle\overline{v'}=\overline {v(a^r)^{t^{-l}}}$ (which includes the case $r=0$).

By Lemma \ref{lem:canCommute} we have  $\displaystyle \overline{[u]_l}=\overline{(a^e)^{t^{-l}}[u]_{l'}}$.
Then \[\begin{array}{llllllll}
\overline u & = & \overline {v(a^{\mathrm{aexp}})^{t^{-l}}[u]_l} 
& = & \overline {v(a^{\mathrm{aexp}})^{t^{-l}}(a^e)^{t^{-l}}[u]_{l'}}\\
\\
& = & \overline {v(a^{\mathrm{aexp}+e})^{t^{-l}}[u]_{l'}}
& = & \overline {v(a^{pq+r})^{t^{-l}}[u]_{l'}}\\
\\
& = & \overline {v(a^r)^{t^{-l}}  (a^{pq})^{t^{-l}}     [u]_{l'}}
& = & \overline {v' (a^{pq})^{t^{-l}}     [u]_{l'}}\\
\\
& = & \overline {v'  t^l (a^{pq}) t^{-l}       [u]_{l'}}
& = & \overline {v'  t^l (ta^{q}t^{-1}) t^{-l}       [u]_{l'}}\\
\\
& = & \overline {v' t^{l+1}(a^{q})t^{-l-1}     [u]_{l'}}
& = & \overline {v' t^{l'}(a^{q})t^{-l'}      [u]_{l'}}\\
\\
& = & \overline {v' (a^{q})^{t^{-l'}}     [u]_{l'}} 
& = & \overline {v'(a^{\mathrm{aexp'}})^{t^{-l'}}[u]_{l'}}.
\end{array}\]
Thus we see that the invariant is preserved.

At the start of the main loop, the loop invariant holds, since $v = \lambda$, $\mathrm{aexp} = 0$,
$l = l_{\min}$, and $u=[u]_{l_{\min}}$.
After the last iteration we have  $l=l_{\max}+1$, $[u]_{l_{\max}+1} = \lambda$ and $\mathrm{aexp} = 0$, so 
$$\overline u =\overline {v(a^{\mathrm{aexp}})^{t^{-l}}[u]_l} =\overline {v(a^{0})^{t^{-l}}\lambda}=\overline v, $$
so the word written to the output
tape is equal in the group to the input word. 
\end{proof}

 \begin{lem}\label{lem:BSalgorithm}
The function $f_S$ can be computed in logspace.
   \end{lem}
\begin{proof} 
At all times  the integers $\mathrm{aexp}$, $|q|$ and $|r|$ are  at most  the length of the input word, so can be computed and  stored in binary in logpsace, and each time the main loop is executed the stored values can be overwritten.  It follows that the algorithm described runs in logspace. 
\end{proof}

 For $u\in S$ define $\beta_u$ to be the value of  the variable $\beta$ stored at the end of the algorithm   computing $f_S(u)$.
 

\begin{cor}\label{cor:negative}
If  $u\in S$ and  $\beta_u<0$, then $f_S(u^{-1})$ contains only $a$ letters.
\end{cor}
\begin{proof}
If $\beta_u<0$ then the top right entry of the matrix representing $u$ is negative, so  the top right entry of the matrix representing $u^{-1}$ is positive, so  $\beta_{u^{-1}}$ is positive, so all $a^{\pm 1}$ letters output by Algorithm \ref{alg:approx} have positive exponent.
\end{proof}



\begin{prop}\label{prop:BSlogspace}
The normal form \begin{itemize}
\item $t^i$,
\item $(a^{\eta_0})^{t^{\alpha_0}} (a^{\eta_1})^{t^{\alpha_1}} \dots (a^{\eta_k})^{t^{\alpha_k}}t^i $,
\item $(a^{-\eta_0})^{t^{\alpha_0}} (a^{-\eta_1})^{t^{\alpha_1}} \dots (a^{-\eta_k})^{t^{\alpha_k}}t^i $,
\end{itemize}
where $i,k \in \Z$, $k\geq 0$, $0< \eta_j <p$,   $\alpha_0 > \alpha_1> \dots > \alpha_k$, and $x^y=y^{-1}xy$,
 can be computed in logspace.
\end{prop}
\begin{proof}

Define a function $h_S$ on (nontrivial) words of the form $$u=(a^{\eta_0})^{t^{\alpha_0}} (a^{\eta_1})^{t^{\alpha_1}} \dots (a^{\eta_{s-1}})^{t^{\alpha_{s-1}}}  (a^{\beta})^{t^{\alpha_s}}$$ with $\beta>0, \alpha_0<\dots<\alpha_{s-1}, 0< \eta_i<p$  as follows. Put $$v=(a^{\eta_0})^{t^{\alpha_0}} (a^{\eta_1})^{t^{\alpha_1}} \dots (a^{\eta_{s-1}})^{t^{\alpha_{s-1}}} ,$$ so $u=v  (a^{\beta})^{t^{\alpha_s}}$. Write $\beta=b_0+b_1p+\dots +b_{\kappa}p^{\kappa}$ with $0\leq b_i<p$ and $b_{\kappa}>0$. Then 
$$h_S(u)=v (a^{b_0})^{t^{\alpha_s}} (a^{b_1})^{t^{\alpha_s-1}}\dots  (a^{b_{\kappa}})^{t^{\alpha_s-\kappa}}.$$

Since $$\overline {a^{\beta}}=\overline{(a^{b_0})(a^{b_1})^{t^{-1}}\dots(a^{b_{\kappa})^{t^{-\kappa}} }}$$
it follows that $\overline u=\overline {h_S(u)}$. The following algorithm
shows that $h_S(u)$ can be computed in logspace. On input $u=v (a^{\beta})^{t^{\alpha_s}}$:
\begin{enumerate}
\item output $v$
\item store $b=\beta$ and $c=\alpha_s$  in binary
\item while  $b>p$: \begin{itemize}\item compute $q,r$ so that $0\leq r<p$ and $b=pq+r$
\item if $r>0$, output $(a^r)^{t^c}$
\item set $c=c-1$ and  $b=q$
\end{itemize}
\item output $(a^{b})^{t^c}$.
\end{enumerate}

Let $\iota:S\rightarrow S$ be the logspace function that computes the inverse of a word given in the proof of Lemma \ref{lem:someExt}.  Define  $\tau:t\mapsto t,t^{-1}\mapsto t^{-1}, a\mapsto a^{-1}$, which is computable with no memory.

We can compute the normal form as follows.
Let $w\in X^*$ be a word written on an input tape.  Run the algorithm in Lemma \ref{lem:step1} to compute $\mathrm{texp}$ and store it in binary.
Run the approximation algorithm on $u=wt^{-\mathrm{texp}}$, suppressing output.
\begin{enumerate}
\item if $\beta_u=0$, output $t^{\mathrm{texp}}$.

\item if  $\beta_u>0$, output $h_S(f_S(u))$, then output  $t^{\mathrm{texp}}$.

\item if  $\beta_u<0$, output $\tau(h_S(f_S(\iota(u)))$, then output  $t^{\mathrm{texp}}$.
\end{enumerate}

\end{proof}

Note that by Proposition \ref{prop:short}, the length of a logspace normal form for an input word of length $n$ is at most a polynomial in $n$. In  this case,  for $p=2$, the input word
$t^{k+1}at^{-k-1}a^{-1}$ of length $2k+4$, has normal form
$\displaystyle aa^ta^{t^2}a^{t^3}a^{t^4}\dots a^{t^k}$ of
length $1+k+\sum_{i=1}^k 2i= k+k(k+1)=k^2+2k+1$.

\section{Logspace embeddable groups}\label{sec:logspaceembeddable}

We define a group to be {\em logspace embeddable} if it embeds in a
group which has a logspace normal form.

Our results from the previous sections give us:
\begin{cor}\label{cor:closurelogspaceembed}
Being logspace embeddable is closed under  direct product, wreath product, and passing to finite index subgroups and supergroups.

\end{cor}

Magnus proved in \cite{Magnus} that a free solvable group can be embedded
in an iterated wreath product of $\Z$. 
(For a modern exposition of this result, see, for example, \cite{MR2645045}.)
Thus we obtain the following corollary of Theorem
\ref{closureWreath}.

\begin{cor}
All finitely generated free solvable groups are logspace embeddable.
In particular, all finitely generated free metabelian groups are logspace embeddable.
\end{cor}

\begin{cor}
If $G$ is logspace embeddable, then the word and co-word problems for $G$ are decidable in logspace (and polynomial time).
\end{cor}
\begin{proof}
By Lemma \ref{lem:inv} we can decide if two words in a group with logspace normal form are equal or not, and moreover by Lemma \ref{lem:polyTime} the algorithm runs is polynomial time. Since the word and coword problems pass to subgroups the result follows.
\end{proof}

While logspace embeddable groups have efficiently decidable word problem,
the same cannot be said for their conjugacy or generalised word problems.
\begin{prop} The generalized word problem and
the conjugacy problem are not decidable for logspace embeddable
groups.
\end{prop}
\begin{proof}
By \cite{\Mihailova} 
 finitely generated subgroups of the direct
product of two finitely generated free groups can have unsolvable
membership problem and unsolvable conjugacy problem.
\end{proof}

We are not able to say whether the class of logspace embeddable groups
is {\em strictly} larger than the class of groups having logspace
normal forms.  If we were able to prove that logspace normal form
implies solvable conjugacy problem, for example, then the proposition above
would settle this.

Further, we know that all linear groups have logspace word problem by
\cite{\LiptonZ}, but we are not able to prove that they all have
logspace normal forms.


\section{Nilpotent groups}\label{sec:nilpotent}

In this section we prove that the group of unitriangular $r \times r$ matrices over $\Z$
has logspace normal form, and obtain as a corollary that all finitely
generated nilpotent groups are logspace embeddable.  It is important
to remember that for our purposes, $r$ is a constant, and in this
respect our algorithms are not uniform.

Let $r \geq 2$, and let $\mbox{UT}_r{\Z}$ be the group of upper
triangular matrices over $\Z$ with $1$'s on the diagonal.  For $1 \leq
i < j \leq r$, let $E_{i,j}$ denote the elementary matrix obtained
from the identity matrix by putting a $1$ in position $(i, j)$ and let
$X$ be the set of all such elementary matrices and their inverses.  $\mbox{UT}_r{\Z}$ is
generated as a group by $X$, and we denote by $\overline{w}$ the image
of $w$ under the natural homomorphism from the free group on $X$ to
$\mbox{UT}_r\Z$.

\begin{lem}
\label{smallEntries}
If $w$ is a word of length $n$ over $X$, and if $a$ is an entry in the
$i$th super-diagonal of the matrix $\overline{w}$, then $|a| \leq
n^i$.
\end{lem}
\begin{proof}
We proceed by induction on $r$.
When $r=2$, $$E_{1,2}=\left[ \begin{array}{cc} 1& 1 \\ 0 & 1 \end{array} \right]$$ and
$\mbox{UT}_2{\Z}\cong \Z$. It is clear that the entry in the off-diagonal  of the matrix $\overline{w}$ has absolute value at most $n$ in this case, so the result holds.

Now let $r\geq 3$ and assume the result holds for  $r-1$.

Consider the homomorphism from
$\mbox{UT}_r \Z$ to $\mbox{UT}_{r-1} \Z$ that takes an $r \times r$
matrix to the $(r-1) \times (r-1)$ matrix in the upper left-hand
corner. The kernel of this homomorphism is isomorphic to $\Z^{r-1}$,
and $\mbox{UT}_r \Z$ is the split extension of $\mbox{UT}_{r-1} \Z$
and this kernel, so with a slight abuse of notation we can consider
$\mbox{UT}_{r-1} \Z$ as a subgroup of $\mbox{UT}_r \Z$.
Furthermore, our chosen generating set is the
disjoint union of generators of the form $E_{i,j}^{\pm 1}$ for $j<r$, and
$E_{i,r}^{\pm 1}$, which generate $\mbox{UT}_{r-1} \Z$ and $\Z^{r-1}$ respectively.  We will denote
by $X_r$ the generating set for $\mbox{UT}_r \Z$.

Let $w$ be a word in $(X_r)^*$ of length $n$.
We now proceed  by induction on $n$.
If $n = 1$, the result is clear, since every entry $a$ in the matrix
$\overline{w}$ satisfies $|a| \leq 1$.  When $n \geq2$, we may assume
that the lemma holds for $n-1$.

  Let $w' $ be a word of length $n-1$
over $X_r$ and let $x$ be an element of $X_r$ such that $w = w' x$.
Then $\overline{w'}$ is of the form
$$
\left[ \begin{array}{cc} A & u \\ 0 & 1 \end{array} \right],
$$
where $A$ is an element of $\mbox{UT}_{r-1} \Z$ and $u$ is a column vector in
 $\Z^{r-1}$.

There are  two cases for the generator $x$: either   $x=E_{i,j}^{\pm 1}$ with $i<j<r$,
that is, $$\left[ \begin{array}{cc} B & 0 \\ 0 & 1 \end{array} \right],$$ where $B$ is of the form $E_{i,j}^{\pm 1}\in X_{r-1}$;
 or $x=E_{i,r}^{\pm 1}$, that is, $$\left[ \begin{array}{cc} I & v \\ 0 & 1 \end{array} \right],$$
where $I$ is the identity matrix and $v$ is a column vector of length $r-1$ with one entry $\pm 1$ and the rest 0.

In the first case, $$\overline{w'}\overline x=\left[ \begin{array}{cc} A & u \\ 0 & 1 \end{array} \right]\left[ \begin{array}{cc} B & 0 \\ 0 & 1 \end{array} \right]=\left[ \begin{array}{cc} AB & u \\ 0 & 1 \end{array} \right].$$
The matrix $AB$ is the product of $n$ generators over $X_{r-1}$, so by inductive assumption on $r$  the entries in the $i$th super-diagonals have absolute value at most $n^i$, and the entries in the last column $$ \left[\begin{array}{c}
u_{r-1} \\ u_{r-2} \\ \vdots \\  u_1  \\1\end{array}\right]$$ satisfy $u_i\le (n-1)^i<n^i$ since they come from $\overline {w'}$.

In the second case,  $$\overline{w'}\overline x=\left[ \begin{array}{cc} A & u \\ 0 & 1 \end{array} \right]\left[ \begin{array}{cc} I & v \\ 0 & 1 \end{array} \right]=\left[ \begin{array}{cc} A & Av+u \\ 0 & 1 \end{array} \right].$$
The entries in the upper left-hand corner satisfy the lemma since they come from $w'$.
Let $a$ be an element of the $k$th superdiagonal of $\overline{w}$, and suppose as well that $a$
is in the last column of $\overline{w}$.
What remains is to show that $|a| < n^k$ in this special case.

The column vector $Av$ is either one of the columns of $A$, or a column of $A$ multiplied by $-1$.
Suppose that $Av$ is the $j$'th column of $A$, or its negation, and denote the entries of $Av$ as follows:
$$ \left[\begin{array}{c}
a_{j-1} \\ a_{j-2} \\ \vdots \\ a_{1} \\ 1 \\0 \\ \vdots \\ 0 \end{array}\right].$$
Denote the entries of $u$ as follows:
$$\left[\begin{array}{c}
u_{r-1} \\ u_{r-2} \\ \vdots \\ u_{1}\end{array}\right].$$
Note that $a_i$ and $u_i$ are on the $i$'th super diagonal of $\overline{w'}$, so by our inductive
assumption on $n$, $|a_i|, |u_i| \leq (n-1)^i$.
Note also that $j \leq r-1$.

There are three cases to consider.
For the first case, let us suppose that $k < r-j$.
In this case $a = u_k + 0 = u_k$ and hence $|a| \leq (n-1)^k < n^k$.
For the second case, let us suppose that $k = r-j$.
In this case, $a = u_k + 1$, and hence $|a| \leq (n-1)^k + 1 \leq n^k$.
In the remaining case, $k > r-j$.
In this case $a = u_k + a_m$, where $m = k - (r-j)$.
Therefore $|a| \leq (n-1)^k + (n-1)^{k-(r-j)}$.
Since $j \leq r-1$, $r-j \geq 1$ and
$$|a| \leq (n-1)^k + (n-1)^{k-1} = (n-1)^{k-1}(n-1+1) = n (n-1)^{k-1} < n^k.$$
\end{proof}

By listing those elementary matrices with a $1$ on the first
superdiagonal first, followed by those elementary matrices with a $1$
on the second superdiagonal next, and so on, we obtain the sequence
\begin{displaymath}
E_{1,2}, E_{2,3}, \dots, E_{r-1,r}, \\ E_{1,3}, E_{2,4}, \dots,
E_{r-2,r}, \\ \dots, \\ E_{1,r},
\end{displaymath}
which is a polycyclic generating sequence for $G$ and hence gives us a
normal form $g_X$ for $G$.

\begin{thm}
The normal form $g_X(w)$ can be computed in logspace.
\end{thm}
\begin{proof}
We begin by describing our algorithm for computing $g_X(w)$.  Compute
and store the matrix $\overline{w}$.  If the entries on the first
super-diagonal are $\alpha_1, \alpha_2, \dots, \alpha_{r-1}$, then
$g_X(w)$ starts with the word $$E_{1,2}^{\alpha_1} E_{2,3}^{\alpha_2}
\dots E_{r-1,r}^{\alpha_{r-1}},$$ so we write this word to the output
tape.  Next compute the matrix for $$w_1 = E_{r-1,r}^{-\alpha_{r-1}}
\dots E_{2,3}^{-\alpha_2} E_{1,2}^{-\alpha_1} w.$$ The matrix
$\overline{w_1}$ will have $0$'s along the first super-diagonal.  Let
$\beta_1$, $\beta_2$, $\dots, \beta_{r-2}$ be the entries on the second
super-diagonal of $\overline{w_1}$.  The next part of $g_x(w)$ starts
with the word
$$E_{1,3}^{\beta_1} E_{2,4}^{\beta_2} \dots
E_{r-2,r}^{\beta_{r-2}},$$ so we write this word to the output tape.
Next compute the matrix for $$w_2 = E_{r-2,r}^{-\beta_{r-2}} \dots
E_{2,4}^{-\beta_2} E_{1,3}^{-\beta_1}w_1.$$ The matrix
$\overline{w_2}$ has $0$'s along the first two super-diagonals.
Continue in this way, peeling off the super-diagonals one at a time,
obtaining at each stage a word $w_i$ such that $\overline{w_i}$ has
$0$'s along the first $i$ super-diagonals, and writing the part of the
normal form $g_X(w)$ that corresponds to the $i$th super-diagonal as
you go.

To show that $g_X(w)$ can be calculated in logspace, it suffices to
show that there exist constants $D$ and $k$ such that for all $1 \leq
i \leq r-1$, the length of $w_i$ is bounded by $Dn^k$, since then by
Lemma \ref{smallEntries}, there exists a constant $C$ such that the
matrix $\overline{w_i}$ can be stored in space $r^2 \log{(C
  (Dn^k)^{r-1}})$, which is $O(\log{n})$.  We will define $D_i$ and
$k_i$ inductively in such a way that for all $i$, the length of $w_i$
is bounded by $D_i n^{k_i}$ and $D_i$ and $k_i$ are constants in the
sense that they do not depend on $n$.  When $i = 0$, $w_i = w$ and we
can take $D_0 = 1$ and $k_0 = 1$.  Now suppose that $D_i$ and $k_i$
are suitable constants for $w_i$.  Let $\beta_1, \beta_2, \dots,
\beta_p$ be the entries on the $(i+1)$-st super-diagonal of
$\overline{w_i}$. By Lemma \ref{smallEntries}, each $\beta_j$ is
bounded in magnitude by $C (D_i n^{k_i})^{r-1}$.  Therefore, the
length of $w_{i+1}$ is bounded by $ p C (D_i n^{k_i})^{r-1} + D_i
n^{k_i} $, which is itself bounded by $ r D_i^r (C + 1)n^{k_i r}$.  We
let $D_{i+1} = r D_i^r (C + 1)$ and $k_{i+1} = k_i r$.  Notice that
neither $D_i$ nor $k_i$ depends on $n$, so from our point of view they
are constants.  Thus, $D$ = $D_{r-1}$ and $k = k_{r-1}$ are constants
such that for all $1 \leq i \leq r-1$, the length of $w_i$ is bounded
b $Dn^k$.
\end{proof}

\begin{cor}
All finitely generated nilpotent groups are logspace embeddable.
\end{cor}
\begin{proof}
Let $G$ be a finitely generated nilpotent group.  Then $G$ has a
finite index subgroup $N$ which is torsion-free (\cite{MR0017281}, Theorem 3.21).
 There exists a positive integer $r$ such that $N$ embeds in $\mbox{UT}_r \Z$
(\cite{\Segal}, Theorem 2, p. 88).  Therefore $N$ is logspace
embeddable.  Since by Proposition \ref{prop:finiteExtQuot} the class of logspace
embeddable groups is closed under finite extension,
   $G$ is logspace embeddable.
\end{proof}

\begin{cor}
All finitely generated groups of polynomial growth are logspace embeddable.
\end{cor}
\begin{proof}
By Gromov's theorem \cite{\Gromov} if a finitely generated group has polynomial growth then it has a  nilpotent finitely generated subgroup of finite index. The result follows from the previous corollary and Corollary \ref{cor:closurelogspaceembed}. \end{proof}

\bibliography{refs} \bibliographystyle{plain}

\end{document}